\documentclass[10pt,english,a4paper]{amsart}
\pdfoutput=1

\usepackage[latin1]{inputenc}
\usepackage[T1]{fontenc}
\usepackage{lmodern}
\usepackage{babel}
\usepackage{amsmath,amssymb,amsthm}
\usepackage[marginratio={1:1,1:1},totalwidth=400pt,totalheight=600pt]{geometry}
\usepackage{microtype}
\usepackage{hyperref}
\usepackage{url}

\newtheorem{theorem}{Theorem}[section]
\newtheorem{lemma}[theorem]{Lemma}
\newtheorem{proposition}[theorem]{Proposition}

\theoremstyle{definition}
\newtheorem{definition}[theorem]{Definition}
\newtheorem{example}[theorem]{Example}
\newtheorem{question}[theorem]{Question}
\newtheorem{remark}[theorem]{Remark}

\theoremstyle{remark}
\newenvironment{delayed-proof}[1][\proofname]{\proof[\bfseries \upshape #1]}{\endproof}

\numberwithin{equation}{section}

\newcommand{\Z}{\mathbb{Z}}

\newcommand{\T}{\mathbb{T}}
\newcommand{\F}{\mathbb{F}}
\newcommand{\C}{\mathbb{C}}

\DeclareMathOperator*{\aut}{Aut}

\newcommand{\op}{\operatorname}

\title[Operator algebras from Artin's representation]{Dynamical~systems and operator~algebras associated to Artin's~representation of braid~groups}

\author[Omland]{Tron Omland}
\address{Department of Mathematics\\University of Oslo\\NO-0316 Oslo\\Norway
\and
Department of Computer Science\\Oslo Metropolitan University\\NO-0130 Oslo\\Norway}
\email{trono@math.uio.no}

\date{April 24, 2019}


\begin{document}

\begin{abstract}
Artin's representation is an injective homomorphism from the braid group $B_n$ on $n$ strands into $\aut\F_n$, the automorphism group of the free group $\F_n$ on $n$ generators. The representation induces maps $B_n\to\aut C^*_r(\F_n)$ and $B_n\to\aut C^*(\F_n)$ into the automorphism groups of the corresponding group $C^*$-algebras of $\F_n$. These maps also have natural restrictions to the pure braid group $P_n$. In this paper, we consider twisted versions of the actions by cocycles with values in the circle, and discuss the ideal structure of the associated crossed products. Additionally, we make use of Artin's representation to show that the braid groups $B_\infty$ and $P_\infty$ on infinitely many strands are both $C^*$-simple.
\end{abstract}

\maketitle

\section{Introduction}

The theory of braids, developed by Artin \cite{Artin}, has been of interest to mathematicians from a wide variety of fields, and provides an intriguing mixture of algebra, analysis, and geometry. Originally, the motivation for studying braid groups came from knot theory, and we refer to the book by Birman \cite{Birman} for an introduction to braids in connection with knots, links, and mapping class groups. Our approach, on the other hand, will be almost purely algebraic.

The braid groups $B_n$ on $n$ strands can be formed from the standard presentation of the symmetric group $S_n$, by removing all order two relations. While $S_n$ is the group of permutations of an $n$-point set, the braid group can also be viewed as being of dynamical nature, acting by permutations twisted by inner automorphisms, where the most natural example to consider is the action of $B_n$ on the free group $\F_n$. This action is defined by an injective homomorphism $B_n\to\aut\F_n$, which is often called ``Artin's representation''. In this paper, we study Artin's representation from an operator algebraic viewpoint, letting instead $B_n$ act on the various group $C^*$-algebras of $\F_n$, and also allowing the permutations to be twisted by a scalar-valued cocycle.

After defining the $C^*$-dynamical systems, we investigate the corresponding crossed product $C^*$-algebras. For the reduced versions, we find conditions for the crossed products to be simple and have a unique tracial state, while for the full versions, we find conditions that ensure primitivity, that is, existence of a faithful irreducible representation.

We organize the paper as follows. First, in Section~\ref{sec:artin-braid}, we recall the definitions and some properties of braid groups and Artin's representation that we will need later, and include a short Section~\ref{sec:prelims} with some preliminaries on group $C^*$-algebras and semidirect products. Then, in Section~\ref{sec:deformed} we introduce the twisted versions of Artin's representations for operator algebras, and state our main theorems, before computing all the possible deformations, that is, the cocycles up to cohomology class. To prove our results, we apply techniques developed in \cite{SUT} for reduced twisted group $C^*$-algebras, especially the so-called ``relative Kleppner condition''. This is described in Section~\ref{sec:proofs}, which also contains the proofs of our main results, relying on several combinatorially involved arguments. In Section~\ref{sec:cohomology} we investigate the second cohomology groups and the twisted group $C^*$-algebras associated to different braid related groups in more detail. Finally, in Section~\ref{sec:infinite} we prove that the braid groups $B_\infty$ and $P_\infty$ on infinitely many strands are both $C^*$-simple, once again by using Artin's representation and the relative Kleppner condition.

\section{Braid groups and Artin's representation}\label{sec:artin-braid}

For every $2\leq n\leq \infty$, the \emph{braid group $B_n$} is defined by generators $s_1,s_2,\dotsc,s_{n-1}$ subject to relations
\[
\begin{gathered}
s_is_{i+1}s_i=s_{i+1}s_is_{i+1} \text{ for all } 1\leq i\leq n-2, \\
s_is_j=s_js_i \text{ when } \lvert i-j \rvert \geq 2.
\end{gathered}
\]
Moreover, for every $n$ there is a natural surjection from $B_n$ onto the symmetric group $S_n$, sending the generator $s_i$ to the permutation $(i,i+1)$, and thus $s_i^2$ to the identity of $S_n$ for all $i$. The \emph{pure braid group $P_n$} is defined as the kernel of the map $B_n \to S_n$.

There is an action $\alpha$ of $B_n$ on the free group $\F_n=\langle x_1,\dotsc,x_n \rangle$ given by
\begin{equation}\label{artin-rep}
\alpha(s_i)(x_j)=
\begin{cases}
x_{i+1} & \text{if } j=i, \\
x_{i+1}^{-1}x_i^{\vphantom{-1}}x_{i+1}^{\vphantom{-1}} & \text{if } j=i+1, \\
x_j & \text{else}.
\end{cases}
\end{equation}
This induces an injective homomorphism $B_n\to\aut\F_n$ called \emph{Artin's representation of $B_n$} (cf.\ \cite[Corollary~1.8.3]{Birman}; note that we use the inverse of their representation, to ensure compatibility with our definition of the semidirect product in Section~\ref{sec:semidirect}, see \cite[Section~1.5.1]{KT}). Remark in particular that $\alpha(s_i)(x_1x_2\dotsm x_n)=x_1x_2\dotsm x_n$ for all $i$. We also use $\alpha$ to denote the restriction of this action to $P_n$. The pure braid groups can now be described as iterative semidirect products by free groups, that is,
\[
P_{n+1} \simeq \F_n \rtimes_\alpha P_n \simeq \F_n \rtimes \F_{n-1} \rtimes \dotsb \rtimes \F_2 \rtimes \Z.
\]
It is possible to describe $P_n$ by generators
\[
a_{i,j} = s_{j-1}\dotsm s_{i+1}s_i^2s_{i+1}^{-1}\dotsm s_{j-1}^{-1}
\]
for $1\leq i<j\leq n$ and certain relations (see \cite[Lemma~1.8.2]{Birman} or \cite[Corollary~1.19]{KT}).
The isomorphism between this version of $P_{n+1}$ and the semidirect product described above
identifies the free group generated by $a_{1,n+1},a_{2,n+1},\dotsc,a_{n,n+1}$ with the normal subgroup $\F_n$.

Moreover, define the \emph{annular braid group $A_n$} as the (non-normal) subgroup of braids of $B_n$ that fix the endpoint of the last string (note that $P_n$ is the subgroup of braids of $B_n$ that fix the endpoint of \emph{all} strings). One can compute that
\[
A_{n+1} = \langle s_1^{\vphantom{2}},\dotsc,s_{n-1}^{\vphantom{2}},s_n^2 \rangle \subseteq B_{n+1}.
\]
The annular braid group was first studied by Chow \cite{Chow}. The group does not seem to have a standardized notation, e.g.\ it is denoted by $D_n$ in \cite[p.~22]{Birman} and by $B_n^1$ in \cite[3.8, 3.9, 4.6]{CoSu}. It is observed in several papers, see e.g.\ \cite[Proposition~2.1]{CrPa}, that
\[
A_{n+1} \simeq \F_n\rtimes_\alpha B_n,
\]
where $a_{1,n+1},\dotsc,a_{n,n+1}\in A_{n+1}$ generate $\F_n$ and $s_1,\dotsc,s_{n-1}\in A_{n+1}$ generate $B_n$.
In fact, the surjections $A_{n+1}\to B_n$ and $P_{n+1}\to P_n$ considered in this section, both with kernel $\F_n$, can be viewed as forgetting the last string.

Clearly, $P_n$ is also a normal subgroup of $A_n$ for all $n\geq 2$, and $A_n/P_n\simeq S_{n-1}$. For completeness, we set $A_1=P_1=B_1=\{e\}$. Thus $[B_n:P_n]=n!$, $[A_n:P_n]=(n-1)!$, and $[B_n:A_n]=n$ for all $n\geq 1$.

Define $\Delta:=s_1(s_2s_1)\dotsm(s_{n-1}\dotsm s_2s_1)$, sometimes called the fundamental element of $B_n$. 
For $n\geq 3$, it is well-known (see e.g.\ \cite[Corollary~1.8.4]{Birman} or \cite[Theorem~1.24, Exercise~1.3.2]{KT}) that the element $z:=\Delta^2$ generates the center of $B_n$, and that
\begin{equation}\label{braid-center}
z=\Delta^2=(s_1\dotsm s_{n-1})^n=a_{12}(a_{13}a_{23})\dotsm (a_{1n}a_{2n}\dotsm a_{n-1,n}).
\end{equation}
This element is a member of both $A_n$ and $P_n$ and generates the centers of those groups too. It is not hard to see that
\begin{equation}\label{center-action}
\alpha(z)(y)=(x_1\dotsm x_n)^{-1}y(x_1\dotsm x_n)
\end{equation}
for all $y\in\F_n$. Indeed, one checks that this holds when $y=x_i$, by applying \eqref{braid-center} and an iterative process using that for all $1\leq i\leq n-1$ we have
\begin{equation}\label{product-action}
\alpha(s_1\dotsm s_{n-1})(x_i)=x_{i+1}\quad\text{and}\quad\alpha(s_1\dotsm s_{n-1})(x_n)=(x_2\dotsm x_n)^{-1}x_1(x_2\dotsm x_n).
\end{equation}
Thus, in the semidirect product description, the generator of the center of $\F_n\rtimes_\alpha B_n$ is
\begin{equation}\label{semidirect-center}
x_1\dotsm x_nz=zx_1\dotsm x_n,
\end{equation}
where $z\in B_n$ is the element described in \eqref{braid-center}, and $\F_n=\langle x_1,\dotsc,x_n \rangle$.

\begin{remark}
\label{rem:ABP}
The group $A_{n+1}$ is isomorphic to the Artin group associated with the irreducible spherical (finite) Coxeter group with matrix entries $m_{n-1,n}=4$, $m_{i,i+1}=3$ when $1\leq i\leq n-2$, and $m_{ij}=2$ else. In the classification system for Coxeter groups, these groups are sometimes classified as type $B$, while the braid groups are of type $A$ (i.e., opposite of the symbols used above for the corresponding Artin groups).
In particular, we have
\begin{equation}\label{A3}
A_3 \simeq \langle t_1,t_2 : (t_1t_2)^2=(t_2t_1)^2 \rangle.
\end{equation}
Moreover, the map $P_3\to\F_2\times\Z=\langle v_1,v_2\rangle\times\langle u\rangle$ given by
\begin{equation}\label{P3}
s_2s_1^2s_2^{-1}\mapsto v_1,\quad s_2^2\mapsto v_2,\quad s_1^2\mapsto (v_1v_2)^{-1}u\quad (\text{and $z\mapsto u$})
\end{equation}
is an isomorphism. In fact, for all $n\geq 1$, we have $P_n \simeq (P_n/Z(P_n)) \times Z(P_n)$.
\end{remark}




\section{Preliminaries}\label{sec:prelims}

\subsection{Group \texorpdfstring{$C^*$}{C*}-algebras}

Here we give some brief preliminaries on group $C^*$-algebras in the slightly more general twisted setting, that we will need later.

Let $G$ be any discrete group with identity $e$. A normalized $2$-cocycle on $G$ is a map $\sigma\colon G\times G\to\T$ satisfying
\[
\begin{gathered}
\sigma(a,b)\sigma(ab,c)=\sigma(a,bc)\sigma(b,c) \\
\sigma(a,e)=\sigma(e,a)=1
\end{gathered}
\]
for all $a,b,c\in G$.

The left regular $\sigma$-projective representation $\lambda_\sigma\colon G\to B(\ell^2(G))$ is defined by
\[
\lambda_\sigma(a)f(b)=\sigma(a,a^{-1}b)f(a^{-1}b).
\]
The reduced twisted group $C^*$-algebra $C^*_r(G,\sigma)$ is the $C^*$-subalgebra of $B(\ell^2(G))$ generated by $\lambda_\sigma(G)$,
and the twisted group von~Neumann algebra $W^*(G,\sigma)$ is the von~Neumann subalgebra of $B(\ell^2(G))$ generated by $\lambda_\sigma(G)$.

The full twisted group $C^*$-algebra $C^*(G,\sigma)$ is the universal $C^*$-algebra generated by a set of unitaries $\{U_\sigma(a)\}_{a\in G}$ subject to the relations $U_\sigma(a)U_\sigma(b)=\sigma(a,b)U_\sigma(ab)$.

\medskip

Let $Z^2(G,\T)$ denote the group of normalized $2$-cocycles on $G$, and let $B^2(G,\T)$ denote its subgroup consisting of all $\sigma\in Z^2(G,\T)$ for which there exists a function $\beta\colon G\to \T$ such that 
\[
\sigma(a,b)=\beta(a)\beta(b)\overline{\beta(ab)} \text{ for all $a,b\in G$}.
\]
The second cohomology group of $G$ with values in $\T$ is then
\[
H^2(G,\T)=Z^2(G,\T)/B^2(G,\T).
\]
We say that $\sigma_1$ and $\sigma_2$ in $Z^2(G,\T)$ are \emph{similar} and write $\sigma_1\sim\sigma_2$ if their image coincide in $H^2(G,\T)$. If this is the case, there is an isomorphism between $C^*_r(G,\sigma_1)$ and $C^*_r(G,\sigma_2)$ given by $\lambda_{\sigma_1}(a)\mapsto\beta(a)\lambda_{\sigma_2}(a)$, where $\beta$ is the function implementing the similarity. The same also holds for the group von~Neumann algebra and for the full group $C^*$-algebra, with the $\lambda_{\sigma_i}(a)$'s replaced by the universal generators $U_{\sigma_i}(a)$ in the latter case.

\subsection{Semidirect products}\label{sec:semidirect}

An action $\alpha$ of a group $K$ on another group $H$ gives rise to a semidirect product $G$ consisting of elements $(x,a)\in H\times K$ and with multiplication given by $(x,a)(y,b)=(x\alpha_a(y),ab)$. As is common, we use the notation $G=H\rtimes_\alpha K$ for a semidirect product. One often identifies $H$ with a normal subgroup of $G$ and $K$ with a subgroup of $G$, and therefore drops the parentheses and just writes $xa$ for an element of $G$, and the action is then given by $\alpha_a(x)=axa^{-1}$.

Moreover, recall that $Z^1(K,\op{Hom}(H,\T))$ consists of all functions $\varphi\colon K\times H \to \T$ satisfying
\begin{equation}\label{1-cocycle}
\begin{gathered}
\varphi(ab,x) = \varphi(a,\alpha_b(x))\varphi(b,x), \\
\varphi(a,xy) = \varphi(a,x)\varphi(a,y).
\end{gathered}
\end{equation}
for $a,b\in K$ and $x,y\in H$. Such functions $\varphi$ are often called $1$-cocycles for $\alpha$ with values in $\T$. The subgroup $B^1(K,\op{Hom}(H,\T))$ of $Z^1(K,\op{Hom}(H,\T))$ consists of all functions $h\colon K\times H \to \T$ for which there exists $f\in\op{Hom}(H,\T)$ such that
\[
h(a,x) = f(\alpha_a(x))\overline{f(x)}
\]
for all $a\in K$ and $x\in H$.
The corresponding cohomology group is then defined as
\[
H^1(K,\op{Hom}(H,\T))=Z^1(K,\op{Hom}(H,\T))/B^1(K,\op{Hom}(H,\T)),
\]
and we write $\varphi_1 \sim \varphi_2$ if $\varphi_1,\varphi_2\in Z^1(K,\op{Hom}(H,\T))$ and $\varphi_1\overline{\varphi_2}\in B^1(K,\op{Hom}(H,\T))$.

\medskip

Let $G=H\rtimes_\alpha K$ be any semidirect product and let $\varphi\in Z^1(K,\op{Hom}(H,\T))$. Define $\sigma^\varphi\in Z^2(G,\T)$ by
\[
\sigma^\varphi((x,a),(y,b))=\varphi(a,y)
\]
for $x,y\in H$ and $a,b\in K$.
For the moment, write $\lambda_\varphi$ for the left regular $\sigma^\varphi$-projective representation of $G$ and $\lambda$ for the left regular representation of $H$.
Then
\[
\lambda_\varphi(a)\lambda_\varphi(x)\lambda_\varphi(a)^*=\sigma^{\varphi}(a,x)\overline{\sigma^{\varphi}(\alpha_{a}(x),a)}\lambda_\varphi(\alpha_a(x))=\varphi(a,x)\lambda(\alpha_a(x)),
\]
so there are $^*$-automorphisms of $C^*_r(H)$ satisfying
\[
\alpha^\varphi_a(\lambda(x))=\varphi(a,x)\lambda(\alpha_a(x)).
\]
It is now an easy exercise to check that $\alpha^\varphi_a\alpha^\varphi_b=\alpha^\varphi_{ab}$ using \eqref{1-cocycle}. Thus, $\alpha^\varphi$ defines an action of $K$ on $C^*_r(H)$, giving rise to a $C^*$-dynamical system for which we can form the reduced crossed product $C^*_r(H) \rtimes_{\alpha^\varphi}^r K$.

Similarly, we get actions of $K$ on $W^*(H)$ and $C^*(H)$, also denoted by $\alpha^\varphi$, producing $W^*$- and full $C^*$-crossed products. It is well-known (see \cite[Theorem~2.1]{Bed}, \cite[Theorem~1]{Bedos}, and \cite[Theorem~4.1]{PR1}, respectively) that
\[
C^*_r(G,\sigma^\varphi)\simeq C^*_r(H) \rtimes_{\alpha^\varphi}^r K, \quad W^*(G,\sigma^\varphi)\simeq W^*(H) \rtimes_{\alpha^\varphi}K,
\]
and
\[
C^*(G,\sigma^\varphi)\simeq C^*(H) \rtimes_{\alpha^\varphi}K.
\]
If $\varphi_1\sim \varphi_2$ via some $f$, define $\beta\colon G\to\T$ by $\beta(x,a)=f(x)$, and then $\alpha^{\varphi_1}\sim\alpha^{\varphi_2}$ via $\beta$.

\subsection{Simplicity, primitivity, Kleppner's condition, and tracial states}

Recall that a $C^*$-algebra is called \emph{simple} it has no nontrivial proper two-sided closed ideals, and \emph{primitive} if it has a faithful irreducible representation. Simplicity clearly implies primitivity, which again implies triviality of the center. A von~Neumann algebra is called a \emph{factor} if it has trivial center (or equivalently, if it does not have any von~Neumann algebra ideals).

A group $G$ is called \emph{icc} if every nontrivial conjugacy class in $G$ is infinite, and this notion generalizes to the twisted setting as follows: an element $a\in G$ is $\sigma$-regular if $\sigma(a,b)=\sigma(b,a)$ whenever $ab=ba$, and $\sigma$-regularity is a property of conjugacy classes, so inspired by \cite{Kle} we say that $(G,\sigma)$ satisfies \emph{Kleppner's condition} if there is no nontrivial finite $\sigma$-regular conjugacy class in $G$.

For a $2$-cocycle $\sigma$ on a countably infinite group $G$, it is explained in \cite[Theorem~2.7]{Om} that $W^*(G,\sigma)$ is a $\textup{II}_1$~factor if and only if $C^*_r(G,\sigma)$ is primitive, if and only if $C^*_r(G,\sigma)$ has trivial center, if and only if $(G,\sigma)$ satisfies Kleppner's condition.

Moreover, a state on a $C^*$-algebra $A$ is a linear functional $\phi\colon A\to\C$ that is positive, i.e., $\phi(S)\geq 0$ whenever $S\in A$ and $S\geq 0$, and unital, i.e., $\phi(1)=1$. A state $\phi$ is called tracial if it satisfies the additional property that $\phi(ST)=\phi(TS)$ for all $S,T\in A$.

There is a canonical faithful tracial state $\tau$ on $C^*_r(G,\sigma)$, namely the vector state associated with $\delta_e$, that is, $\tau(S)=\langle S\delta_e,\delta_e\rangle$ for all $S\in C^*_r(G,\sigma)$.

A group $G$ is called \emph{$C^*$-simple} if $C^*_r(G)$ simple and is said to have \emph{the unique trace property} if $\tau$ is the only tracial state on $C^*_r(G)$. It is shown in \cite[Theorem~1.3]{BKKO}, based on \cite{KK}, that $C^*$-simplicity is stronger than the unique trace property (in fact, strictly stronger by \cite{Boudec,Iva-Om}). If $G$ is $C^*$-simple (resp.\ has the unique trace property), then $C^*_r(G,\sigma)$ is simple (resp.\ has a unique tracial state) for every $\sigma\in Z^2(G,\T)$, cf.\ \cite[Corollaries~4.5 and~5.3]{BK}. In general, for a pair $(G,\sigma)$ with $\sigma\not\sim 1$, it is not known whether simplicity of $C^*_r(G,\sigma)$ implies that $\tau$ is its only tracial state.

Note that in \cite{FCH,SUT} a pair $(G,\sigma)$ is said to be $C^*$-simple (resp.\ have the unique trace property) when $C^*_r(G,\sigma)$ is simple (resp.\ has a unique tracial state).

\begin{remark}\label{FC-Z}
The group $B_n/Z(B_n)$ is $C^*$-simple for all $n\geq 1$, according to \cite[p.~536]{Bed}. Thus its normal subgroup $P_n/Z(P_n)$ is also $C^*$-simple for all $n\geq 1$, by applying \cite[Theorem~1.4]{BKKO}. Moreover, $A_n/Z(A_n)$ embeds into $B_n/Z(B_n)$ as a subgroup of finite index, and is therefore $C^*$-simple for all $n\geq 1$, by \cite[p.~216]{BdH}.

As a consequence, all these quotient groups are icc, which by \cite[Proposition~2.5]{FCH} means that every conjugacy class in $B_n$, $P_n$, or $A_n$ is either infinite or a one-element set.

Finally, we also note that for any of these groups, the center coincides with the amenablish radical defined in \cite[Section~6]{Iva-Om}.
\end{remark}

\section{Twisted versions of Artin's representation}\label{sec:deformed}

We now apply Section~\ref{sec:semidirect} to Artin's representation, i.e., with $H=\F_n$ and $K=B_n$ or $P_n$. 

\begin{definition}\label{AB-def}
Let $\varphi\colon B_n \times \F_n \to \T$ be an element of $Z^1(B_n,\op{Hom}(\F_n,\T))$, and let $\alpha$ be as in \eqref{artin-rep}. Define the reduced $\varphi$-deformed version of Artin's representation of braid goups as the $C^*$-dynamical system $\alpha^\varphi\colon B_n\to\aut C^*_r(\F_n)$ given on generators by
\[
\alpha^\varphi_{s_j}(\lambda(x_k)) = \varphi(s_j,x_k)\lambda(\alpha_{s_j}(x_k)).
\]
The von~Neumann algebra version $\alpha^\varphi\colon B_n\to\aut W^*(\F_n)$ is defined analogously.
\end{definition}
Let $\mu^\varphi\in\T$ be the product of the values of $\varphi$ on generators, i.e.,
\[
\mu^\varphi=\prod_{\substack{1\leq j\leq n-1\\1\leq k\leq n}}\varphi(s_j,x_k).
\]
As usual, an element $w\in\T$ is called nontorsion if $w^m\neq 1$ for all $m\in\Z\setminus\{0\}$.
\begin{theorem}\label{AB-theorem}
Let $\alpha$, $\varphi$, $\alpha^\varphi$, and $\mu^\varphi$ be as above.
Then the following are equivalent:
\begin{itemize}\itemsep2pt
\item[(i)] $\mu^\varphi$ is nontorsion,
\item[(ii)] $C^*_r(\F_n)\rtimes_{\alpha^\varphi}^r B_n$ is simple,
\item[(iii)] $C^*_r(\F_n)\rtimes_{\alpha^\varphi}^r B_n$ has a unique tracial state,
\item[(iv)] $W^*(\F_n)\rtimes_{\alpha^\varphi}B_n$ is a factor.
\end{itemize}
\end{theorem}

The \hyperref[proof-AB]{proof} of Theorem~\ref{AB-theorem} is given below in Section~\ref{sec:proofs}.

\begin{definition}\label{P-def}
Let $\varphi\colon P_n \times \F_n \to \T$ be an element of $Z^1(P_n,\op{Hom}(\F_n,\T))$, and let $\alpha$ be as in \eqref{artin-rep}. Define the reduced $\varphi$-deformed version of Artin's representation of pure braid goups as the $C^*$-dynamical system $\alpha^\varphi\colon P_n\to\aut C^*_r(\F_n)$ given by
\[
\alpha^\varphi_{a_{ij}}(\lambda(x_k)) = \varphi(a_{ij},x_k)\lambda(\alpha_{a_{ij}}(x_k)).
\]
The von~Neumann algebra version $\alpha^\varphi\colon P_n\to\aut W^*(\F_n)$ is defined analogously.
\end{definition}
For $1\leq k\leq n$, let $\nu_k^\varphi\in\T$ be the value
\[
\nu^\varphi_k=\prod_{1\leq i<j\leq n}\varphi(a_{ij},x_k).
\]
\begin{theorem}\label{P-theorem}
Let $\alpha$, $\varphi$, $\alpha^\varphi$, and $\nu^\varphi_k$ be as above.
Consider the following conditions:
\begin{itemize}\itemsep2pt
\item[(i)] $\nu^\varphi_k$ is nontorsion for at least one $k$,
\item[(ii)] $C^*_r(\F_n)\rtimes_{\alpha^\varphi}^r P_n$ is simple,
\item[(iii)] $C^*_r(\F_n)\rtimes_{\alpha^\varphi}^r P_n$ has a unique tracial state,
\item[(iv)] $W^*(\F_n)\rtimes_{\alpha^\varphi}P_n$ is a factor.
\end{itemize}
Then (i) $\Longrightarrow$ (ii) $\Longrightarrow$ (iv) and (i) $\Longrightarrow$ (iii) $\Longrightarrow$ (iv).

Moreover, if $n=2$, then (iv) $\Longrightarrow$ (i).
\end{theorem}

The \hyperref[proof-P]{proof} of Theorem~\ref{P-theorem} is given below in Section~\ref{sec:proofs}.

\begin{example}\label{ex1}
Let us consider the case $n=2$, where $B_2\simeq\Z$ is generated by $s=s_1$ and $P_2\simeq\Z$ is generated by $a=a_{12}=s^2$. First, for the braid group we have
\[
\begin{gathered}
\alpha^\varphi_s(\lambda(x_1)) = \varphi(s,x_1)\lambda(x_2), \\
\alpha^\varphi_s(\lambda(x_2)) = \varphi(s,x_2)\lambda(x_2^{-1}x_1^{\vphantom{-1}}x_2^{\vphantom{-1}}).
\end{gathered}
\]
Moreover, $\varphi_1\sim \varphi_2$ if and only if $\varphi_1(s,x_1)\varphi_1(s,x_2)=\mu^{\varphi_1}=\mu^{\varphi_2}=\varphi_2(s,x_1)\varphi_2(s,x_2)$, as explained in Proposition~\ref{prop:1-cocycles} below.
For the pure braid group, we have
\[
\begin{gathered}
\alpha^\varphi_a(\lambda(x_1)) = \varphi(a,x_1)\lambda(x_2^{-1}x_1^{\vphantom{-1}}x_2^{\vphantom{-1}}), \\
\alpha^\varphi_a(\lambda(x_2)) = \varphi(a,x_2)\lambda(x_2^{-1}x_1^{-1}x_2^{\vphantom{-1}}x_1^{\vphantom{-1}}x_2^{\vphantom{-1}}),
\end{gathered}
\]
In particular, we notice that $\alpha_a=\op{Ad}x_2^{-1}x_1^{-1}$ in this case.
See Example~\ref{ex2} for more on this.
\end{example}

\begin{question}
In Theorem~\ref{P-theorem}, does (iv) $\Longrightarrow$ (i), (ii), or (iii) for $n\geq 3$?
\end{question}

\begin{definition}
Let $\alpha$ be as in \eqref{artin-rep}.
First, for $\varphi\in Z^1(B_n,\op{Hom}(\F_n,\T))$, the full $\varphi$-deformed version of Artin's representation of braid groups is the map $\alpha^\varphi\colon B_n\to\aut C^*(\F_n)$, defined on generators similarly as in Definition~\ref{AB-def}.

Secondly, for $\varphi\in Z^1(P_n,\op{Hom}(\F_n,\T))$, the full $\varphi$-deformed version of Artin's representation of pure braid groups is the map $\alpha^\varphi\colon P_n\to\aut C^*(\F_n)$, defined on generators similarly as in Definition~\ref{P-def}.
\end{definition}

\begin{theorem}\label{primitivity}
The full crossed product $C^*(\F_2)\rtimes_{\alpha^\varphi} B_2$ is primitive if and only if it has trivial center, if and only if $\mu^\varphi=\varphi(s,x_1)\varphi(s,x_2)$ is nontorsion.

The full crossed product $C^*(\F_2)\rtimes_{\alpha^\varphi} P_2$ is primitive if and only if it has trivial center, if and only if at least one of $\nu^\varphi_1=\varphi(a,x_1)$ and $\nu^\varphi_2=\varphi(a,x_2)$ is nontorsion.
\end{theorem}

\begin{proof}
Assume $\mu=\mu^\varphi$ is nontorsion. We will first show that $C^*(\F_2)\rtimes_{\alpha^\varphi} B_2$ is primitive. According to work of Choi \cite{Choi}, there exists a faithful irreducible representation $\pi$ of $C^*(\F_2)$ on some separable Hilbert space $\mathcal{H}$. Let $U_1$ and $U_2$ be the two generators of $C^*(\F_2)$ and set $V_1=\pi(U_1)$ and $V_2=\pi(U_2)$. We may assume that $V_2$ is diagonal with distinct diagonal entries (see \cite[Proof of Theorem~6]{Choi}). For every $z\in\T$ define an automorphism $\gamma_z$ on $C^*(\F_2)$ by $\gamma_z(U_1)=U_1$ and $\gamma_z(U_2)=zU_2$, and set $\pi_z=\pi\circ\gamma_z$. Clearly, $\pi_z$ is also faithful and irreducible for every choice of $z\in\T$, and we would like to show that $\pi_z\circ\alpha^\varphi_{s^k}\not\simeq\pi_z$ for every $k\in\Z\setminus\{0\}$, and then apply \cite[Theorem~2.1]{BO-pams}. We compute that
\[
\pi_z\circ\alpha^\varphi_{s^{2k}}(U_i)=z\mu^k\pi(U_1U_2)^{-k}\pi(U_i)\pi(U_1U_2)^k\simeq z\mu^k\pi(U_i)
\]
for $i=1,2$ all $k\in\Z$, and moreover that
\[
\pi_z\circ\alpha^\varphi_{s^{2k+1}}(U_1)=\pi_z\circ\alpha^\varphi_{s^{2k}}(\varphi(s,x_1)U_2)\simeq \varphi(s,x_1)z\mu^k\pi(U_2)
\]
for all $k\in\Z$. Applying the same strategy as in \cite[proof of Theorem~2.3]{BO-pams}, using that the point spectrum of $V_2$ is countable, it is not hard to see that
\[
\Omega:=\bigcup_{k\in\Z} \left(\{z:V_2\not\simeq z\mu^k V_2\} \cup \{z:V_1\not\simeq \varphi(s,x_1)z\mu^k V_2\}\right)
\]
is a countable set. Thus, by choosing $z\in\T\setminus\Omega$, we get that $\pi_z\circ\alpha^\varphi_{s^k}\not\simeq\pi_z$ for every $k\in\Z\setminus\{0\}$. 

Hence, it follows from \cite[Theorem~2.1]{BO-pams} that $C^*(\F_2)\rtimes_{\alpha^\varphi} B_2$ is primitive.

All the other implications follow from known results after identifying the crossed products with twisted group $C^*$-algebras as explained in Example~\ref{ex2} below.
\end{proof}

\begin{question}
For $n\geq 3$, does condition~(i) of Theorem~\ref{AB-theorem} and~\ref{P-theorem} imply primitivity of $C^*(\F_n)\rtimes_{\alpha^\varphi} B_n$ and $C^*(\F_n)\rtimes_{\alpha^\varphi} P_n$, respectively?
\end{question}


Before we prove the above results in the next section, we compute the cohomology groups:

\begin{proposition}\label{prop:1-cocycles}
We have
\[
H^1(B_n,\op{Hom}(\F_n,\T))\simeq \begin{cases}
\T & \text{for $n=2$}, \\
\T^2 & \text{for $n\geq 3$},
\end{cases}
\]
where the cohomology class $[\varphi]$ of $\varphi\in Z^1(B_n,\op{Hom}(\F_n,\T))$ is determined by the parameters \eqref{mu1} and \eqref{mu2} below.

Moreover, for all $n\geq 2$, we have
\[
H^1(P_n,\op{Hom}(\F_n,\T)) = Z^1(P_n,\op{Hom}(\F_n,\T)) = \op{Hom}(P_n,\op{Hom}(\F_n,\T))\simeq\T^{\frac{1}{2}n(n-1)\cdot n}.
\]
\end{proposition}

\begin{proof}
Let $n\geq 2$ and pick $\varphi\in Z^1(B_n,\op{Hom}(\F_n,\T))$.
Computations using \eqref{1-cocycle} and the fact that $\varphi(s_is_{i+1}s_i,x_j)=\varphi(s_{i+1}s_is_{i+1},x_j)$ for all $1\leq i\leq n-2$ and $1\leq j\leq n$ show the following:
By taking $j=i$ or $i+2$ we get that
\begin{equation}\label{cocycle-relation-1}
\varphi(s_{i+1},x_i)=\varphi(s_i,x_{i+2}) \text{ for all $1\leq i\leq n-2$}.
\end{equation}
By letting $j=i+1$ and using \eqref{cocycle-relation-1} we get that
\begin{equation}\label{cocycle-relation-2}
\varphi(s_i,x_i)\varphi(s_i,x_{i+1})=\varphi(s_{i+1},x_{i+1})\varphi(s_{i+1},x_{i+2}) \text{ for all $1\leq i\leq n-2$},
\end{equation}
that is, $\varphi(s_i,x_i)\varphi(s_i,x_{i+1})$ takes the same value for all $1\leq i\leq n-1$.
Finally, we get that
\begin{equation}\label{cocycle-relation-3}
\varphi(s_i,x_j)=\varphi(s_{i+1},x_j) \text{ for all $1\leq i\leq n-2$ and all $j\notin\{i,i+1,i+2\}$}.
\end{equation}
If $n\geq 3$, then computations using \eqref{1-cocycle} and the fact that $\varphi(s_is_j,x_k)=\varphi(s_js_i,x_k)$ whenever $\lvert i-j\rvert\geq 2$ and $1\leq k\leq n$ show that
\begin{equation}\label{cocycle-relation-4}
\varphi(s_i,x_k)=\varphi(s_i,x_\ell) \text{ for all $1\leq i\leq n-2$ and all $k,\ell\notin\{i,i+1\}$}.
\end{equation}
By combining \eqref{cocycle-relation-1}, \eqref{cocycle-relation-2}, \eqref{cocycle-relation-3}, and \eqref{cocycle-relation-4}, we see that there is $\mu_1\in\T$ such that
\begin{equation}\label{mu1}
\mu_1=\varphi(s_i,x_i)\varphi(s_i,x_{i+1}) \text{ for all $1\leq i\leq n-1$},
\end{equation}
and if $n\geq 3$, there is $\mu_2\in\T$ such that
\begin{equation}\label{mu2}
\mu_2=\varphi(s_i,x_j) \text{ for all $1\leq i\leq n-1$ and all $j\notin\{i,i+1\}$}.
\end{equation}
These two values, together with a choice of $\varphi(s_i,x_i)$ or $\varphi(s_i,x_{i+1})$ for every $i$, determine $\varphi$.
Further computations show that $B^1(B_n,\op{Hom}(\F_n,\T))$ consists of all functions $h$ satisfying
\[
h(s_i,x_i)h(s_i,x_{i+1})=h(s_i,x_j)=1 \text{ for all $1\leq i\leq n-1$ and all $j\notin\{i,i+1\}$}.
\]
Indeed, if $h(a,x)=f(\alpha_a(x))\overline{f(x)}$ for some $f\in\op{Hom}(\F_n,\T)$, then
\[
h(s_i,x_{i+1})=f(x_i)\overline{f(x_{i+1})}=\overline{h(s_i,x_i)} \text{ for all $1\leq i\leq n-1$},
\]
and $h(s_i,x_j)=f(x_i)\overline{f(x_i)}=1$ if $j\notin\{i,i+1\}$.
Therefore, for $n\geq 3$ the values $\mu_1,\mu_2\in\T$ determine the class of $\varphi$ in $H^1(B_n,\op{Hom}(\F_n,\T))$, and two distinct such choices give rise to nonsimilar $1$-cocycles. For $n=2$, the class of $\varphi$ is determined by the single parameter $\mu_1$.

The last statement follows from a fairly straightforward calculation, showing that for all $a\in P_n$, there is some $x\in\F_n$ such that $\alpha(a)=\op{Ad}x$.
That is, the action $\alpha$ of $P_n$ on $\op{Hom}(\F_n,\T)$ is trivial.
\end{proof}

\begin{lemma}
Let $\varphi\in Z^1(B_n,\op{Hom}(\F_n,\T))$ and $z$ is the generator of $Z(B_n)$ defined in \eqref{braid-center}. For every $1\leq i\leq n$ and $1\leq j\leq n-1$, we have
\begin{equation}\label{eq:z-relation}
\mu^\varphi=\varphi(z,x_i)=\varphi(s_j,x_1\dotsm x_n)^{n-1}.
\end{equation}
Moreover, we note that $z$ acts trivially on $\op{Hom}(\F_n,\T)$, so that
\begin{equation}\label{varphi-z}
\varphi(za,x)=\varphi(z,x)\varphi(a,x) \text{ and } \varphi(z,x)=\varphi(z,\alpha_a(x)) \text{ for all $a\in B_n$ and $x\in\F_n$}.
\end{equation}
\end{lemma}

\begin{proof}
First, by combining \eqref{cocycle-relation-1} and \eqref{cocycle-relation-2} we have that
\[
\varphi(s_i,x_ix_{i+1}x_{i+2})=\varphi(s_{i+1},x_ix_{i+1}x_{i+2}) \text{ for all $1\leq i\leq n-2$}.
\]
Then, by using \eqref{cocycle-relation-3} we see that $\varphi(s_i,x_1\dotsm x_n)$ takes the same value for all $1\leq i\leq n-1$, namely the value $\mu_1^{\vphantom{n-2}}\mu_2^{n-2}$, where $\mu_1,\mu_2$ are the values described in the proof of Proposition~\ref{prop:1-cocycles}.

Calculations using \eqref{1-cocycle} and \eqref{product-action} show that
\[
\varphi((s_1,\dotsc,s_{n-1})^n,x_i)=\varphi(s_1,\dotsc,s_{n-1},x_1\dotsm x_n),
\]
for all $i$, so by \eqref{braid-center} and the fact that all $s_i$ act trivially on $x_1\dotsm x_n$, we get that
\[
\varphi(z,x_i)=\prod_{\substack{1\leq j\leq n-1\\1\leq k\leq n}}\varphi(s_j,x_k)=(\mu_1^{\vphantom{n-2}}\mu_2^{n-2})^{n-1},
\]
and hence \eqref{eq:z-relation} holds. Showing \eqref{varphi-z} is straightforward from \eqref{1-cocycle} and \eqref{center-action}.
\end{proof}

\section{Proofs via twisted group \texorpdfstring{$C^*$}{C*}-algebras}\label{sec:proofs}

Let $H$ be a normal subgroup of a discrete group $G$, and let $\sigma$ be a $2$-cocycle on $G$. Following \cite[Definition~4.5]{SUT}, we say that $g\in G$ is $\sigma$-regular with respect to $H$ if $\sigma(g,x)=\sigma(x,g)$ whenever $x\in H$ and $gx=xg$. Moreover, $(G,H,\sigma)$ satisfies \emph{the relative Kleppner condition} if $C_H(g)=\{xgx^{-1}:x\in H\}$ is infinite whenever $g\in G\setminus H$ is $\sigma$-regular with respect to $H$.

We refer to \cite[Section~4]{SUT} for a brief discussion on the connection between the relative Kleppner condition and freely acting (or equivalently, properly outer) automorphisms of the twisted group von~Neumann algebra. In this paper we will repeatedly apply the following:

\begin{proposition}[{\cite[Corollary~4.11]{SUT}}]\label{SUT-rkc}
Let $G$ be a discrete group with a normal subgroup $H$. Let $\sigma$ be a $2$-cocycle on $G$, and denote its restriction to $H$ by $\sigma'$.
Assume moreover that $(G,H,\sigma)$ satisfies the relative Kleppner condition.

If $C^*_r(H,\sigma')$ is simple (resp.\ has a unique tracial state), then $C^*_r(G,\sigma)$ is simple (resp.\ has a unique tracial state).
\end{proposition}

To prove Theorem~\ref{AB-theorem} and~\ref{P-theorem} we will identify the crossed products with reduced twisted group $C^*$-algebras as explained in Section~\ref{sec:semidirect}. We first need the following result:

\begin{lemma}\label{braid-centralizer}
Suppose that $g\in \F_n\rtimes_\alpha B_n$.
Then $C_{\F_n}(g)$ is finite if and only $g\in Z(\F_n\rtimes_\alpha B_n)$.
\end{lemma}

\begin{proof}
Assume that $C_{\F_n}(g)$ is finite. It is explained in \cite[Proof of Proposition~4.18]{SUT} that $C_{\F_n}(g)=\{g\}$, but for the convenience of the reader we sketch the argument.
Let $g'\in C_{\F_n}(g)$, so $g' = xgx^{-1}$ for some $x \in \F_n$.
Then we have $g^{-1}g'=(g^{-1}xg)x^{-1}\in\F_n$.
Moreover, $C_{\F_n}(g^{-1}g') \subseteq C_{\F_n}(g)^{-1} C_{\F_n}(g') = C_{\F_n}(g)^{-1} C_{\F_n}(g)$, so
\[
\lvert C_{\F_n}(g^{-1}g')\rvert \leq \lvert C_{\F_n}(g)^{-1} C_{\F_n}(g')\rvert \leq \lvert C_{\F_n}(g)\rvert^2 < \infty.
\]
Since $\F_n$ is icc, we must have $g^{-1}g'=e$. Thus, $g'=g$, that is, $C_{\F_n}(g)=\{g\}$, and it follows that $g$ belongs to the centralizer of $\F_n$ in $\F_n\rtimes_\alpha B_n$.

Next, let $y\in\F_n$ and $a\in B_n$ so that $g=ya\in\F_n\rtimes_\alpha B_n$. Since $g$ belongs to the centralizer, it commutes with $x_i$ for all $i$, i.e., $yax_ia^{-1}y^{-1}=x_i$, that is, $\alpha(a)(x_i)=y^{-1}x_iy$ for all $i$. Then
\[
x_1\dotsm x_n=\alpha(a)(x_1\dotsm x_n)=y^{-1}x_1\dotsm x_ny,
\]
meaning that $y$ commutes with $x_1\dotsm x_n$, and thus $y=(x_1\dotsm x_n)^k$ for some $k\in\Z$ (recall that the centralizer of a single element in a free group is a cyclic subgroup). This again means by \eqref{center-action} that $\alpha(a)=\alpha(z^k)$, that is, $a=z^k$ because of injectivity of $\alpha$. Hence, by \eqref{semidirect-center},
\[
g=ya=(x_1\dotsm x_n)^kz^k=(x_1\dotsm x_nz)^k\in Z(\F_n\rtimes_\alpha B_n).\qedhere
\]
\end{proof}
The analogous result and proof also holds for $P_n$ in place of $B_n$.

\begin{lemma}\label{kleppner-AB}
Let $\varphi\in Z^1(B_n,\op{Hom}(\F_n,\T))$. The following are equivalent:
\begin{itemize}
\item[(i)] $\varphi(z,x_j)$ is nontorsion for some $1\leq j\leq n$,
\item[(ii)] $(\F_n\rtimes_\alpha B_n,\F_n,\sigma^\varphi)$ satisfies the relative Kleppner condition,
\item[(iii)] $(\F_n\rtimes_\alpha B_n,\sigma^\varphi)$ satisfies Kleppner's condition.
\end{itemize}
\end{lemma}

\begin{proof}
Assume that $\varphi(z,x_j)$ is nontorsion for some $1\leq j\leq n$.
Suppose that $g\notin\F_n$ and $C_{\F_n}(g)$ is finite.
By Lemma~\ref{braid-centralizer} we have that $g\in Z(\F_n\rtimes_\alpha B_n)$, so \eqref{semidirect-center} gives that $g=(x_1\dotsm x_n)^kz^k$ for some $k\in\Z\setminus\{0\}$. Then, using \eqref{varphi-z}
\[
\sigma^\varphi(g,x_j)\overline{\sigma^\varphi(x_j,g)}=\varphi(z^k,x_j)=\varphi(z,x_j)^k\neq 1.
\]
Hence, $(\F_n\rtimes_\alpha B_n,\F_n,\sigma^\varphi)$ satisfies the relative Kleppner condition.

Assume next that $\varphi(z,x_j)$ is torsion for all $j$, that is, there exists $k\in\Z\setminus\{0\}$ such that $\varphi(z^k,x_j)=\varphi(z,x_j)^k=1$ for all $j$ by \eqref{varphi-z}. Set
\[
g=(x_1\dotsm x_n)^{k(n-1)}z^{k(n-1)} \in Z(\F_n\rtimes_\alpha B_n).
\]
We will show that $g$ is $\sigma^\varphi$-regular, i.e., that $\sigma^\varphi(g,ya)=\sigma^\varphi(ya,g)$ for all $y\in\F_n$ and all $a\in B_n$.
Note that
\[
\sigma^\varphi(g,ya)\overline{\sigma^\varphi(ya,g)}=\varphi(z^{k(n-1)},y)\overline{\varphi(a,(x_1\dotsm x_n)^{k(n-1)})}.
\]
Since $z$ is central, $\varphi(z^{k(n-1)},y)=\varphi(z^k,y)^{n-1}=1$.
By applying \eqref{eq:z-relation} and \eqref{varphi-z}, we get that
\[
\varphi(s_i,x_1\dotsm x_n)^{k(n-1)}=\varphi(z^k,x_1\dotsm x_n)=1
\]
for all $i$, so $\varphi(a,x_1\dotsm x_n)^{k(n-1)}=1$ for all $a\in B_n$ as well.
Thus $g$ is $\sigma^\varphi$-regular, so $(\F_n\rtimes_\alpha B_n,\sigma^\varphi)$ does not satisfy Kleppner's condition.

We skip the proof of (ii)~$\Longrightarrow$~(iii) here, since it follows from the \hyperref[proof-AB]{proof} of Theorem~\ref{AB-theorem}.
\end{proof}

\begin{delayed-proof}[Proof of Theorem~\ref{AB-theorem}]\phantomsection\label{proof-AB}
We identify the two algebras in question with $C^*_r(\F_n\rtimes_\alpha B_n,\sigma^\varphi)$ and $W^*(\F_n\rtimes_\alpha B_n,\sigma^\varphi)$, and recall that the latter is a factor if and only if $(\F_n\rtimes_\alpha B_n,\sigma^\varphi)$ satisfies Kleppner's condition (cf.\ \cite{Kle}). Then (iv)~$\Longrightarrow$~(i) immediately follows from Lemma~\ref{kleppner-AB}.

Next, (i)~$\Longrightarrow$~(ii) and (i)~$\Longrightarrow$~(iii) follow from Lemma~\ref{kleppner-AB} and Proposition~\ref{SUT-rkc}, together with the well-known fact that $C^*_r(\F_n)$ is simple and has a unique trace (cf.\ \cite{Powers}).

Finally, (ii)~$\Longrightarrow$~(iv) and (iii)~$\Longrightarrow$~(iv) holds as Kleppner's condition is necessary both for simplicity and uniqueness of trace, see \cite[Section~2.3]{SUT} (this also takes care of the last implication from Lemma~\ref{kleppner-AB}).
\end{delayed-proof}

\begin{lemma}\label{kleppner-P}
Let $\varphi\in Z^1(P_n,\op{Hom}(\F_n,\T))$. Consider the following conditions:
\begin{itemize}
\item[(i)] $\varphi(z,x_k)$ is nontorsion for some $1\leq k\leq n$,
\item[(ii)] $(\F_n\rtimes_\alpha P_n,\F_n,\sigma^\varphi)$ satisfies the relative Kleppner condition,
\item[(iii)] $(\F_n\rtimes_\alpha P_n,\sigma^\varphi)$ satisfies Kleppner's condition,
\item[(iv)] $\varphi(a_{ij},x_1\dotsm x_n)$ or $\varphi(z,x_k)$ is nontorsion for some $1\leq i<j\leq n$ or $1\leq k\leq n$.
\end{itemize}
Then (i) $\Longleftrightarrow$ (ii) $\Longrightarrow$ (iii) $\Longleftrightarrow$ (iv).
\end{lemma}

\begin{proof}
First, (i)~$\Longrightarrow$~(ii) follows by the same argument as in Lemma~\ref{kleppner-AB}, applying Lemma~\ref{braid-centralizer}.
For the converse, if $\varphi(z,x_i)$ is torsion for all $i$, we can find $\ell\in\Z\setminus\{0\}$ such that $\varphi(z,x_i)^\ell=1$ for all $i$, and then choose $g=(x_1\dotsm x_n)^\ell z^\ell$. Then $C_{\F_n}(g)=\{g\}$ and $\sigma^\varphi(g,x)=\sigma^\varphi(x,g)$ for all $x\in\F_n$.

By Remark~\ref{FC-Z}, to determine Kleppner's condition, it is enough to deal with central elements, i.e., of the form $g=(x_1\dotsm x_n)^\ell z^\ell$ for $\ell\in\Z$.
We note that
\[
\overline{\sigma^\varphi(g,a_{ij})}\sigma^\varphi(a_{ij},g)=\varphi(a_{ij},(x_1\dotsm x_n)^\ell)=\varphi(a_{ij},x_1\dotsm x_n)^\ell
\]
for all $1\leq i<j\leq n$ and
\[
\sigma^\varphi(g,x_k)\overline{\sigma^\varphi(x_k,g)}=\varphi(z^\ell,x_k)=\varphi(z,x_k)^\ell
\]
for all $1\leq k\leq n$.
Hence, we conclude that (iii)~$\Longrightarrow$~(iv), and then (i)~$\Longrightarrow$~(iv) is obvious.
\end{proof}

\begin{delayed-proof}[Proof of Theorem~\ref{P-theorem}]\phantomsection\label{proof-P}
This goes along similar lines as the \hyperref[proof-AB]{proof} of Theorem~\ref{AB-theorem}, except we do not have that Kleppner's condition implies the relative Kleppner condition for $n\geq 3$. For $n=2$, see Example~\ref{ex2}.
\end{delayed-proof}

\begin{example}\label{ex2}
When $n=2$, then $(\F_2\rtimes_\alpha B_2,\sigma^\varphi)$ satisfies Kleppner's condition precisely when $\mu=\mu^\varphi$ is nontorsion.
This completes the proof of first statement of Theorem~\ref{primitivity} since primitivity implies triviality of the center, which again implies Kleppner's condition by \cite[Corollary~2.8]{Om}.

For the second statement of Theorem~\ref{primitivity}, recall from \eqref{P3} that $\F_2\rtimes_\alpha P_2\simeq\F_2\times\Z$, so
\[
C^*(\F_2)\rtimes_{\alpha^\varphi} P_2 \simeq C^*(\F_2 \times \Z,\sigma)
\]
for some $2$-cocycle $\sigma$ on $\F_2 \times \Z$. This $C^*$-algebra is discussed in \cite[Example~3.11]{Om}, where it is shown to be primitive if and only if at least one of $\nu_1=\varphi(a,x_1)$, $\nu_2=\varphi(a,x_2)$ is nontorsion, which is again equivalent with Kleppner's condition.

Finally, the above also takes care of the very last statement of Theorem~\ref{AB-theorem}.
\end{example}

\section{More on twisted group \texorpdfstring{$C^*$}{C*}-algebras and the second cohomology group}\label{sec:cohomology}

The nontrivial $2$-cocycles on $B_n$ are the ones lifted from $S_n$, that is,
\[
H^2(B_n,\T)=H^2(S_n,\T)=\begin{cases} 1 & \text{if } n\leq 3, \\ \Z/2\Z & \text{if } n\geq 4.\end{cases}
\]
Let $n\geq 4$ and $\sigma$ denote a nontrivial $2$-cocycle on $B_n$.
The universal $C^*$-algebra generated by a set of unitaries $\{U_i\}_{i=1}^{n-1}$ subject to the relations
\[
\begin{gathered}
U_iU_{i+1}U_i=U_{i+1}U_iU_{i+1} \text{ for all } 1\leq i\leq n-2, \\
U_iU_j=-U_jU_i \text{ when } \lvert i-j \rvert \geq 2.
\end{gathered}
\]
is isomorphic to $C^*(B_n,\sigma)$.
Since the image of $Z(B_n)$ is trivial under the quotient map $B_n \to S_n$, every element of $Z(B_n)$ is $\sigma$-regular.
It follows that $(B_n,\sigma)$ never satisfies Kleppner's condition, so $C^*(B_n,\sigma)$ and $C^*_r(B_n,\sigma)$ both have nontrivial centers.

\begin{lemma}\label{PL second cohomology}
The second cohomology groups of $A_{n+1}$ and $P_{n+1}$ are
\[
\begin{gathered}
H^2(A_{n+1},\T) \simeq H^2(B_n,\T) \times H^1(B_n,\op{Hom}(\F_n,\T)),\\
H^2(P_{n+1},\T) \simeq H^2(P_n,\T) \times H^1(P_n,\op{Hom}(\F_n,\T)),
\end{gathered}
\]
and every $2$-cocycle on $P_{n+1}$ is similar to one given by
\begin{equation}\label{eq:mackey}
\sigma(xa,yb)=\varphi(a,y)\omega(a,b) \text{ for $a,b\in P_n$ and $x,y\in\F_n$},
\end{equation}
where $\varphi\in Z^1(P_n,\op{Hom}(\F_n,\T))$ and $\omega$ is a $2$-cocycle on $P_n$.

Moreover, if $\sigma_1$ and $\sigma_2$ are $2$-cocycles on $P_{n+1}$ of the form \eqref{eq:mackey},
then $\sigma_1 \sim \sigma_2$ if and only if $\varphi_1 \sim \varphi_2$ and $\omega_1 \sim \omega_2$.

The last two statements also hold when replacing $P_{n+1}$ and $P_n$ by $A_{n+1}$ and $B_n$.
\end{lemma}

\begin{proof}
Since $H^2(\F_n,\T)$ is trivial for all $n$, the first statement follows directly from the Lyndon-Hochschild-Serre spectral sequence (see e.g.\ \cite[p.~715]{PR}).
The second statement is deduced from \cite[Theorem~9.4]{Mac}, and the third statement is explained in \cite[2.1-2.4]{Om2} (see also \cite[Appendix~2]{PR}).
\end{proof}
It now follows from Lemma~\ref{PL second cohomology} that
\[
H^2(A_3,\T)\simeq\T, \quad H^2(A_4,\T)\simeq\T^2, \quad H^2(A_n,\T)\simeq\T^2\oplus\Z_2 \text{ for } n\geq 5,
\]
and by applying some basic summation formulas, we get
\[
H^2(P_n,\T)\simeq\T^{\frac{1}{24}n(n-1)(n-2)(3n-1)}.
\]

\begin{proposition}
Let $n\geq 3$ and $\sigma\in Z^2(A_n,\T)$. Then the following are equivalent:
\begin{itemize}
\item[(i)] $(A_n,\sigma)$ is $C^*$-simple,
\item[(ii)] $(A_n,\sigma)$ has the unique trace property,
\item[(iii)] $(A_n,\sigma)$ satisfies Kleppner's condition.
\end{itemize}
In particular, $A_n$ belongs to the class $\mathcal{K}$ introduced in \cite{SUT}.
\end{proposition}

\begin{proof}
Let $n\geq 4$ and $\sigma$ be a $2$-cocycle on $A_{n+1}$ of the form \eqref{eq:mackey}, coming from a pair $(\varphi,\omega)$.
We proceed as in the proof of Lemma~\ref{kleppner-AB}, choosing $g$ the same way, and compute that
\[
\sigma(g,ya)\overline{\sigma(ya,g)}=\varphi(z^{k(n-1)},y)\omega(z^{k(n-1)},a)\overline{\varphi(a,(x_1\dotsm x_n)^{k(n-1)})\omega(a,z^{k(n-1)})}.
\]
It is observed above that all elements of $Z(B_n)$ are $\omega$-regular, so $\omega(z^{k(n-1)},a)=\omega(a,z^{k(n-1)})$, meaning that the $\omega$ disappears, and the rest of the argument goes exactly as in Lemma~\ref{kleppner-AB} and the \hyperref[proof-AB]{proof} of Theorem~\ref{AB-theorem}.
\end{proof}

\begin{example}
Let $\sigma$ be a $2$-cocycle on $A_3$.
By Proposition~\ref{PL second cohomology} and~\ref{prop:1-cocycles}, we may, up to similarity, assume that $\sigma=\sigma^\varphi$, where $\varphi(s,x_1)=1$ and $\sigma(s,x_2)=\mu$ for some $\mu\in\T$. By using \eqref{A3}, it is straightforward to check that $C^*(A_3,\sigma)$ is the universal $C^*$-algebra generated by unitaries $U_1,U_2$ subject to the relation
\[
(U_1U_2)^2=\mu(U_2U_1)^2,
\]
and by Theorem~\ref{primitivity}, this algebra is primitive if and only if $\mu$ is nontorsion.
\end{example}

\begin{proposition}
Let $\sigma\in Z^2(P_{n+1},\T)$ be of the type described in \eqref{eq:mackey}, coming from a pair $(\varphi,\omega)$.
Consider the following conditions:
\begin{itemize}
\item[(i)] $\varphi(z,x_k)$ is nontorsion for some $1\leq k\leq n$,
\item[(ii)] $(P_{n+1},\sigma)$ is $C^*$-simple,
\item[(iii)] $(P_{n+1},\sigma)$ has the unique trace property,
\item[(iv)] at least one of the following hold:
\begin{itemize}
\item[] $\varphi(a_{ij},x_1\dotsm x_n)\omega(a_{ij},z)\overline{\omega(z,a_{ij})}$ is nontorsion for some $1\leq i<j\leq n$,
\item[] $\varphi(z,x_k)$ is nontorsion for some $1\leq k\leq n$.
\end{itemize}
\end{itemize}
Then (i)~$\Longrightarrow$~(ii)~$\Longrightarrow$~(iv) and (i)~$\Longrightarrow$~(iii)~$\Longrightarrow$~(iv).

Moreover, if $n=3$, then (iv)~$\Longrightarrow$~(i).
\end{proposition}

\begin{proof}
A modification of Lemma~\ref{kleppner-P} shows that (iv) is equivalent to Kleppner's condition, while (i) is equivalent to the relative Kleppner condition.
The argument now goes along the same lines as the \hyperref[proof-P]{proof} of Theorem~\ref{P-theorem}.
\end{proof}

\section{\texorpdfstring{$C^*$}{C*}-simplicity of braid groups on infinitely many strands}\label{sec:infinite}

The braid group on infinitely many strands is the group $B_\infty$ generated by $\{s_i\}_{i=1}^\infty$ subject to the same relations as defining $B_n$. Moreover, there are embeddings $B_n\to B_{n+1}$, given by $s_i\mapsto s_i$ for $1\leq i\leq n-1$, and it is easy to see that $B_\infty$ coincides with the direct limit of this system. For $m<n\leq\infty$, we will consider $B_m$ as the subgroup of $B_n$ coming from the first $m$ strands.

Recall from \eqref{braid-center} that for $3\leq n<\infty$, the center of $B_n$ is generated by an element of length $n(n-1)$, and all generators $s_i$ of $B_n$ are needed to write word that is a central element. Thus, if $m<n$, then every nontrivial element of $Z(B_m)$ does not belong to $Z(B_n)$, that is,
\begin{equation}\label{central-intersection}
Z(B_m) \cap Z(B_n) = \{e\}.
\end{equation}
Therefore, $B_\infty$ has trivial center. In fact, all of its nontrivial conjugacy classes are infinite, since if $g$ has a finite conjugacy class in $B_\infty$, then there is some $N\geq 3$ such that $g\in B_n$ and $g$ has a finite conjugacy class in $B_n$ for all $n\geq N$. Then $g\in Z(B_n)$ for all $n\geq N$, so $g=e$. Thus $W^*(B_\infty)$ is a factor (see \cite[Corollary~5.3]{GoKo} for a few more properties of $W^*(B_\infty)$).

Furthermore, let $S_\infty$ denote the infinite symmetric group, and define the pure braid group on infinitely many strands, $P_\infty$, as the kernel of the surjection $B_\infty \to S_\infty$, sending $s_i^2$ to the identity for all $i$. Analogously to the above, it is not hard to see that $P_\infty$ coincides with the direct limit of the embeddings $P_n\to P_{n+1}$ given by $a_{ij}\mapsto a_{ij}$ for $1\leq i<j\leq n$.

Finally, we define Artin's representation $\alpha\colon B_\infty \to \aut\F_\infty$ similarly as in \eqref{artin-rep}. 

\begin{lemma}
The semidirect products $\F_\infty\rtimes_\alpha B_\infty$ and $\F_\infty\rtimes_\alpha P_\infty$ are both $C^*$-simple.
\end{lemma}

\begin{proof}
The triple $(\F_\infty\rtimes_\alpha B_\infty,\F_\infty,1)$ always satisfies the relative Kleppner condition by modifying the argument of Lemma~\ref{braid-centralizer}. Indeed, if there is $g\in\F_\infty\rtimes_\alpha B_\infty$ such that $C_{\F_\infty}(g)$ is finite, then $C_{\F_\infty}(g)=\{g\}$, that is, $g=ya$ for $y\in\F_\infty$, $a\in B_\infty$, and $gx_i=x_ig$ for all $i$, where $\{x_i\}_{i=1}^\infty$ are the generators of $\F_\infty$. There is some $N\geq 3$ such that $g\in \F_n\rtimes_\alpha B_n$ for all $n\geq N$, meaning that $C_{\F_n}(g)=\{g\}$ for all $n\geq N$, so $g\in Z(\F_n\rtimes_\alpha B_n)$ for all $n\geq N$ by Lemma~\ref{braid-centralizer}. Then \eqref{semidirect-center} says that we must have $a\in Z(B_n)$ for all $n\geq N$, so $a=e$ by \eqref{central-intersection}, and hence $g\in\F_\infty$. 

Therefore, it follows from Proposition~\ref{SUT-rkc} that $\F_\infty\rtimes_\alpha B_\infty$ is $C^*$-simple.
Since $\F_\infty\rtimes_\alpha P_\infty$ is normal in $\F_\infty\rtimes_\alpha B_\infty$, it is also $C^*$-simple by \cite[Theorem~1.4]{BKKO}.
\end{proof}

\begin{theorem}
The groups $P_\infty$ and $B_\infty$ are both $C^*$-simple.
\end{theorem}

\begin{proof}
First, consider the surjection $P_\infty \to P_\infty$ given by $a_{1,i}\mapsto e$ and $a_{i,j}\mapsto a_{i-1,j-1}$ for all $2\leq i<j$. This map restricts to surjections $P_{n+1}\to P_n$ with kernel being the free group $\F_n$ generated by $\{a_{1,j}\}_{j=2}^{n+1}$ (see \cite[Corollary~1.23]{KT}; this is just Artin's representation with the generators in reversed order). These restrictions give rise to short exact sequences $\F_n \to P_{n+1} \to P_n$ compatible with the embeddings $P_n\to P_{n+1}$, so $P_\infty\to P_\infty$ becomes its direct limit, and the kernel is the free group $\F_\infty$ on the infinite number of generators $\{a_{1,j}\}_{j=2}^\infty$. Thus we have that $P_\infty \simeq \F_\infty \rtimes_\alpha P_\infty$ via the map $a_{1,i}\mapsto x_{i-1}$ and $a_{i,j}\mapsto a_{i-1,j-1}$ for all $2\leq i<j$, and the latter group is shown to be $C^*$-simple in the previous lemma.

To show that $B_\infty$ is $C^*$-simple, we will explain that $(B_\infty,P_\infty,1)$ satisfies the relative Kleppner condition, and then use Proposition~\ref{SUT-rkc} once more to conclude.

So suppose that $g\in B_\infty$ is such that $C_{P_\infty}(g)$ is finite. Arguing as in Lemma~\ref{braid-centralizer}, using that $P_\infty$ is icc (since it is $C^*$-simple), we get that $C_{P_\infty}(g)=\{g\}$, i.e., $gx=xg$ for all $x\in P_\infty$. Moreover, there is some $N\geq 3$ such that $g\in B_n$ for all $n\geq N$, meaning that $C_{P_n}(g)=\{g\}$ for all $n\geq N$. Since $P_n$ has finite index in $B_n$, for every $n\geq N$, there is some $k(n)\geq 1$ such that $g^{k(n)}\in P_n$, so if $g$ commutes with all elements of $P_n$, then $g^{k(n)}\in Z(P_n)$. Therefore, if $N\leq m<n$, then $g^{k(m)k(n)}\in Z(P_m)\cap Z(P_n)=\{e\}$, using \eqref{central-intersection}. Since $B_n$ is torsion-free for all $n$ (see e.g.\ \cite[Corollary~1.29]{KT}), we must have $g=e$. This completes the proof.
\end{proof}


\section*{Acknowledgements}

The author is funded by the Research Council of Norway through FRINATEK, project no.~240913. Part of the work was done when attending the research program ``Classification of operator algebras: complexity, rigidity, and dynamics'' at the Mittag-Leffler Institute outside Stockholm, January--February 2016. Finally, the author would like to thank Fred Cohen, Alex Suciu, and Dan Cohen for helpful e-mail correspondence on braid groups.

\bibliographystyle{plain}

\end{document}